\documentclass[a4paper,10pt,reqno]{amsart}
\pdfoutput = 1
\usepackage[utf8]{inputenc}
\usepackage[T1]{fontenc}
\usepackage{lmodern}
\usepackage[draft=false,kerning=true]{microtype}

\usepackage{a4wide}

\usepackage{amsfonts,amsmath,amssymb}

\usepackage{graphicx}

\usepackage{ifthen}

\usepackage{tikz}
\usetikzlibrary{decorations.pathreplacing}

\usepackage{url}
\usepackage{cite}

\usepackage{hyperref}

\newtheorem{thm}{Theorem}

\newtheorem{prop}{Proposition}
\newtheorem{lem}[prop]{Lemma}

\theoremstyle{remark}

  \theoremstyle{plain}
  \newtheorem{beh}{}
  
  \newcommand{\behref}[1]{{\bfseries \ref{#1}}}

\numberwithin{equation}{section}

\newcommand{\Z}{\mathbb{Z}}

\newcommand{\Se}{\mathcal{S}}
\newcommand{\Oh}{\mathcal{O}}
\renewcommand{\Re}{\operatorname{Re}}
\renewcommand{\Im}{\operatorname{Im}}
\newcommand{\Arg}{\operatorname{Arg}}

\newcommand{\OEIS}[1]{#1}

\newcommand{\TODO}[1]%
{\par\fbox{\begin{minipage}{0.9\linewidth}\textbf{TODO:} #1\end{minipage}}\par}

\usepackage{suffix}
\newcommand{\parentheses}[4][]%
{\ifthenelse{\equal{#1}{}}{\left#2}{\csname#1\endcsname#2}%
    {#4}\ifthenelse{\equal{#1}{}}{\right#3}{\csname#1\endcsname#3}}

\newcommand{\foperator}[1]{\ensuremath{%
    \mathop{{#1}\thinspace\negthinspace}
    \mathchoice{\negthinspace}{\negthinspace}{}{}}}

\newcommand{\f}[3][]{\ensuremath{\foperator{#2}\parentheses[#1]{(}{)}{#3}}}
\WithSuffix\newcommand{\f}*[2]{\ensuremath{
    \mathop{{#1}\thinspace\negthinspace}{#2}}}

\newcommand{\abs}[2][]{\ensuremath{%
    \parentheses[#1]{\lvert}{\rvert}{#2}}}

\newcommand{\dni}[2][]{\ensuremath{%
    \parentheses[#1]{\lVert}{\rVert}{#2}}}  

\newcommand{\fexp}[2][]{\f[#1]{\exp}{#2}}
\newcommand{\flog}[2][]{\f[#1]{\log}{#2}}
\newcommand{\fOh}[2][]{\f[#1]{\Oh}{#2}}

  \newlength{\wwwoolen}

\begin{document}




\title[Multi-Base Representations of Integers]{Multi-Base Representations of
  Integers: Asymptotic Enumeration and Central Limit Theorems}


\author{Daniel Krenn}
\address{Daniel Krenn \\
Institute of Analysis and Computational Number Theory (Math A) \\
Graz University of Technology \\
Steyrergasse 30 \\
8010 Graz \\
Austria
}
\email{\href{mailto:math@danielkrenn.at}{math@danielkrenn.at} \textit{or}
  \href{mailto:krenn@math.tugraz.at}{krenn@math.tugraz.at}}

\thanks{Daniel Krenn is supported 
    by the Austrian Science Fund (FWF): P24644,
    by the Austrian Science Fund (FWF): F5510, which is part of the 
    Special Research Program "`Quasi-Monte Carlo Methods:
    Theory and Applications"',
    by the Austrian Science Fund (FWF): I1136, and
    by the Austrian Science Fund (FWF): W1230, 
    Doctoral Program ``Discrete Mathematics''.}


\author{Dimbinaina Ralaivaosaona}
\address{Dimbinaina Ralaivaosaona \\
Department of Mathematical Sciences \\
Stellenbosch University \\
Private Bag X1 \\
Matieland 7602 \\
South Africa
}
\email{\href{mailto:naina@sun.ac.za}{naina@sun.ac.za}}
\thanks{Dimbinaina Ralaivaosaona is supported by the Subcommittee B of Stellenbosch University, South Africa.
}

\author{Stephan Wagner}
\address{Stephan Wagner \\
Department of Mathematical Sciences \\
Stellenbosch University \\
Private Bag X1 \\
Matieland 7602 \\
South Africa
}
\email{\href{mailto:swagner@sun.ac.za}{swagner@sun.ac.za}}

\thanks{Stephan Wagner is supported by the National Research
  Foundation of South Africa under grant number 70560.}


\date{\today}




\begin{abstract}
  In a multi-base representation of an integer (in contrast to, for example,
  the binary or decimal representation) the base (or radix) is replaced by
  products of powers of single bases. The resulting numeral system has desirable properties for fast arithmetic.
 It is usually redundant, which means that each integer can have multiple different digit
  expansions, so the natural question for the number of representations arises. In this paper, we provide a general asymptotic formula for the number of
 such  multi-base representations of a positive integer $n$. Moreover, we prove   central limit theorems for the sum of digits, the Hamming weight (number of non-zero digits, which is a measure of efficiency) and the occurrences of a fixed digits in a  random representation.
\end{abstract}

\maketitle


\section{Introduction and Background}


A \emph{numeral system}\footnote{We use the term \emph{numeral system} rather
  than \emph{number system} as it is also called sometimes, since that name is
  ambiguous. For example, the system of $p$-adic numbers or the system of real
  numbers are called number systems.} (also called \emph{system of numeration})
is a way to represent numbers. The most common examples are, of course, the
ordinary decimal and binary systems, which represent numbers in base~$10$ and
$2$, respectively. Besides those ``standard'' systems, there is an immense
number of other numeral systems.

For fast arithmetic, the right choice of numeral system is an important
aspect. The algorithms we have in mind here are, for example, exponentiation in
a finite group and the scalar multiplication on elliptic curves. Both are used
in cryptography, and clearly we want to improve on the running time of those
algorithms (which are often based on a Horner scheme, cf.\@
Knuth~\cite{Knuth:1998:Art:2}).

Starting with the binary system, one can improve the performance of the
aforementioned algorithms by adding more digits than needed. Thus, we make
the numeral system \emph{redundant}, which means that each element can have a
lot of different representations. For instance, using digits $0$, $1$ and $-1$
can lead to a speed-up, cf.\@ Morain and Olivos~\cite{Morain-Olivos:1990} for
such a scalar multiplication algorithm on elliptic curves. To gain back the
uniqueness, additional syntax can complement the redundant system. In the
example using digits $0$, $1$ and $-1$, this can be the non-adjacent form, see
Reitwiesner's seminal paper~\cite{Reitwiesner:1960}. Generalizations in that
direction can be found in \cite{Blake-Seroussi-Smart:1999, Gordon:1998,
  Miyaji-Ono-Cohen:1997:effic, Solinas:2000:effic-koblit}.


A different way to get redundancy, and thereby a better running time of the
algorithms mentioned above, is to use double-base and multi-base numeral
systems. For example, we can represent a number by a finite sum of terms
$a_\ell\, 3^{\alpha_\ell} 7^{\beta_\ell} 11^{\gamma_\ell}$ for
some digits~$a_\ell$, which leads to a multi-base system with three bases. A
formal definition is given in the next section. Note that multiplication by one of the bases
(in the example: $3$, $7$ or $11$) is extremely simple for such representations, just like doubling is easy for binary
representations. This is a very desirable property for fast arithmetic.

Double-base numeral systems are used for cryptographic applications, see for
example \cite{Avanzi-Dimitrov-Doche-Sica:2006:double-base,
  Dimitrov-Imbert-Mishra:2008:double-base,
  Dimitrov-Jullien-Miller:1999:double-base}. The typical bases are~$2$
and~$3$. With these bases (and a digit set containing at least $0$ and $1$),
each positive integer has a double-base representation, cf.\@ Berth\'e and
Imbert~\cite{Berthe-Imbert:2009}. When using general bases, less is known on
the existence, cf.\@ Krenn, Thuswaldner and
Ziegler~\cite{Krenn-Thuswaldner-Ziegler:2013:belcher} for some results using
small symmetric digit sets. However, choosing the digit set large enough (so
that the numeral system with only one of the bases can already represent all
positive integers), existence can always be guaranteed. Thus, when each
positive integer has a multi-base representation, a natural further question arises---
and this is also the main question studied in this article: how many
representations does each integer have? Our Theorem~\ref{thm:asy-general} provides
an (asymptotic) answer to this question.

The question is also motivated by the cryptoanalysis of evaluation schemes
(e.g.\@ elliptic curve scalar multiplication): One can avoid side channel
attacks if the corresponding numeral system is very redundant, i.e., if each
element has many different representations. In addition to the number
of representations, other parameters, such as the (Hamming) weight or the sum of digits, are
of importance in this context and therefore studied here as well. The Hamming weight in particular
is a measure for the efficiency of a digit representation for fast arithmetic. We show here that the sum of digits
and the Hamming weight (as well as the number of occurrences of any fixed digit) of a typical representation
of $n$ is of order $(\log n)^m$, where $m$ is the number of bases.

Our paper is structured as follows. The following section provides more precise definitions and reviews
existing results on the number of representations (which are available in very special cases). This is followed (in
Section~\ref{sec:main-results}) by the precise statements of our main
results. These results also include, apart from the asymptotic enumeration of multi-base
representations, the analysis of the sum of digits, the (Hamming) weight and the number of occurrences of a fixed digit.
The remaining parts of this article (Sections~\ref{sec:generating-function} to~\ref{sec:parameters}) are devoted
to the proofs of all these results, which are based on generating functions and the  saddle-point method.
Section~\ref{sec:conclusion} concludes the paper.

An extended abstract of this paper was presented at the AofA 2014 conference in
Paris, see~\cite{Krenn-Ralaivaosaona-Wagner:2014:multi-base-asy}.


\section{Terminology and Existing Results}
\label{sec:existing}


In a \emph{multi-base representation of $n$} (or \emph{multi-base
  expansion}), a positive integer~$n$ is expressed as a finite sum
\begin{equation}\label{eq:def-multi-base-repr}
  n = \sum_{\ell=1}^L  a_\ell B_\ell, \tag{$\divideontimes$}
\end{equation}
such that the following holds.
\begin{itemize}
\item The $a_\ell$ (called \emph{digits}) are taken from a fixed finite
  \emph{digit set} $\mathcal{D}$. Here, we will be using the canonical digit set
  $\{0,1,\dots,d-1\}$ for some fixed integer $d \geq 2$, but in principle our methods work for other sets as well.
\item The $B_\ell$ are in increasing order (i.e., $B_1 < B_2 < \cdots < B_L$) and
  taken from the set
  \begin{equation*}
    \Se=\lbrace p_1^{\alpha_1}p_2^{\alpha_2}\dots p_m^{\alpha_m} : \alpha_i \in \mathbb{N}\cup \{0\} \rbrace.
  \end{equation*}
  The numbers $p_1$, \dots, $p_m$ are called the \emph{bases} (in our setting, these
  are fixed coprime integers greater than $1$). The sequence of all elements of $\Se$ in increasing order is sometimes called a \emph{Hardy--Littlewood--P\'olya-sequence}.
\end{itemize}

In the following, we will discuss the number of representations of $n$ in a
given multi-base system, which we denote by $P(n)$ (we suppress the dependence
on $p_1$, $p_2$, \dots, $p_m$ and $d$). Note that this number is finite, since
our digit set does not contain negative integers.


For redundant single-base representations a lot is known.
Reznick~\cite{Reznick:1990:some} presents results on certain partition
functions, which correspond to representations with non-negative digits; see
also Protasov~\cite{Protasov:2000:asy-partitions,
  Protasov:2004:probl-asy-partitions} for more recent results on the number of
representations $P(n)$. When negative digits are used as well (for example in
elliptic curve cryptography), there are usually infinitely many representations
of a number, so counting these does not make sense. In this case, expansions
with minimum number of non-zero digits are of interest, since they lead to
fast evaluation schemes. See Grabner and
Heuberger~\cite{Grabner-Heuberger:2006:Number-Optimal} for a result counting
minimal representations (one minimal representation is the non-adjacent form
mentioned above, cf.\@ also \cite{Heuberger-Krenn:2013:wnaf-analysis,
  Heuberger-Krenn:2013:wnafs-optimality, Reitwiesner:1960}).


Let us consider double-base systems in particular, and let us take bases~$2$ and~$p$, where $p>1$
is an odd integer, and digits $0$ and $1$. We can group terms involving the
same powers of $p$ and use the uniqueness of the binary expansion to show that
double-base representations with bases~$2$ and~$p$ are in bijection with
partitions into powers of~$p$, i.e., representations of the form
\begin{equation*}
  n = n_0 + n_1 p + n_2 p^2 + n_3 p^3 + \cdots
\end{equation*}
with (arbitrary) non-negative integers $n_\ell$. More generally,
the same is true for
double-base representations with bases $q$ and $p$ and digit set
$\{0,1,\ldots,q-1\}$. It seems that the first non-trivial approximation of
$P(n)$ in this special case is due to
Mahler~\cite{Mahler:1940:spec-functional-eq}. By studying Mordell's functional
equation, he obtained
\begin{equation*}
  \log P(pn) \sim (\log n)^2 / (2\log p).
\end{equation*}
The much more precise result 
\begin{align*}
  \log P(pn) &= \frac{1}{2\log p}\left(\log\frac{n}{\log n}\right)^2
  + \left( \frac12 + \frac{1}{\log p} + \frac{\log\log p}{\log p} \right)
  \log n \\
  &\mathrel{\phantom{=}}
  - \left( 1 + \frac{\log\log p}{\log p} \right) \log\log n
  + \fOh{1}
\end{align*}
was derived by Pennington~\cite{Pennington:1953:Mahler-part-prob}. The error
term in the previous asymptotic formula exhibits a periodic fluctuation. Note
that for bases $2$ and $p$, the function $P(n)$ fulfils the recurrence
relation
\begin{equation*}
 P(n) =
 \begin{cases}
   P(n-1) + P(n/p) &\text{if $p\mid n$,} \\
   P(n-1) &\text{otherwise,}
 \end{cases}
\end{equation*}
which has been known for a long time in conjunction with partitions of integers.

For further reference and more information see \OEIS{A005704} in the On-Line
Encyclopedia of Integer Sequences~\cite{OEIS:2015} and see also
\cite{Dimitrov-Imbert-Mishra:2008:double-base,
  Mishra-Dimitrov:2008:comb-double-base} for the connection to double-base
systems.


\section{Main Results}
\label{sec:main-results}


We present our main results now.
The aim of this work is to give an asymptotic formula in a more general set-up.
Throughout this paper, $d\geq 2$ and $m\geq 2$ are fixed integers, and $p_1$,
$p_2$, \dots, $p_m$ are integers such that $1< p_1< p_2< \dots <p_m$ and
$\gcd(p_i,p_j)=1$ for $i\neq j$. As our first main theorem, we prove an
asymptotic formula for the number of representations of $n$ of the
form~\eqref{eq:def-multi-base-repr}. It will be convenient to use the
abbreviation
\begin{equation*}
  \kappa=\frac{\log d}{m!}\prod_{j=1}^{m}\frac{1}{\log p_j}.
\end{equation*}

\begin{thm}\label{thm:asy-general}
  If $m\geq 3$, then the number $P(n)$ of distinct multi-base representations
  of $n$ of the form~\eqref{eq:def-multi-base-repr} satisfies the asymptotic
  formula
\begin{equation*}
  \log P(n)= C_0 (\log n)^m+C_1 (\log n)^{m-1}\log \log n+C_2 (\log n)^{m-1}
  +\fOh{(\log n)^{m-2}\log\log n}
\end{equation*}
for $n\to\infty$, where
\begin{align*}
  C_0 &= \kappa,\\
  C_1 &= -m(m-1) \kappa,\\
  C_2 &= \kappa m\Big(1+\tfrac{1}{2}\sum_{j=1}^{m} 
  \log p_j-\tfrac{1}{2}\log d -\log(\kappa m)\Big).
\end{align*}
\end{thm}

In the case that there are precisely two bases, we have the following more
precise asymptotic result.

\begin{thm}\label{thm:asy-2}
  If $m = 2$, then the number $P(n)$ of distinct multi-base representations of
  $n$ of the form~\eqref{eq:def-multi-base-repr} satisfies the asymptotic
  formula
\begin{equation*}
  P(n) = K(n) (\log n)^{K_0}n^{K_1}
  \fexp{\kappa \log^2 \left( \frac{n}{\log n} \right)}
\end{equation*}
for $n\to\infty$, where $K(n)$ is a fluctuating function of $n$, that is,
bounded above and below by positive constants, and
\begin{align*}
  K_0 &= \tfrac{1}{2}+2\kappa\big(\log(2\kappa)
  -\tfrac{1}{2}(\log p_1 +\log p_2-\log d)\big),\\
  K_1 &= 2\kappa\big(1-\log(2\kappa)
  +\tfrac{1}{2}(\log p_1 +\log p_2-\log d)\big)-1.
\end{align*}
\end{thm}

Note that the first two terms of the asymptotic formula in
Theorem~\ref{thm:asy-general} coincide with those in Theorem~\ref{thm:asy-2}.

Moreover, we study the distribution of three natural parameters in random
multi-base representations, namely the sum of digits, i.e.\@ $a_1 + a_2 + \dots
+ a_L$ in the notation of~\eqref{eq:def-multi-base-repr}, the Hamming weight
(the number of non-zero coefficients $a_\ell$) and the number of occurrences of a fixed digit $b$. We get the following theorems.

\begin{thm}\label{thm:sum_of_digits}
  The sum of digits in a random multi-base representation of $n$ of the
  form~\eqref{eq:def-multi-base-repr} asymptotically follows a Gaussian
  distribution with mean and variance equal to
  \begin{equation*}
    \mu_n = \frac{\kappa(d-1)}{2 \log d} (\log n)^{m}
    +\fOh{(\log n)^{m-1}\log\log n}
  \end{equation*}
  and
  \begin{equation*}
    \sigma_n^2 = \frac{\kappa(d-1)(d+1)}{12 \log d}
    (\log n)^{m}+\fOh{(\log n)^{m-1}\log\log n}
  \end{equation*}
  respectively.
\end{thm}

\begin{thm}\label{thm:hamming}
  The Hamming weight of a random multi-base representation of $n$ of the
  form~\eqref{eq:def-multi-base-repr} asymptotically follows a Gaussian
  distribution with mean and variance equal to
  \begin{equation*}
    \mu_n = \frac{\kappa(d-1)}{d \log d}
    (\log n)^{m}+\fOh{(\log n)^{m-1}\log\log n} 
  \end{equation*}
  and
  \begin{equation*}
    \sigma_n^2 = \frac{\kappa(d-1)}{d^2 \log d} (\log n)^{m}
    +\fOh{(\log n)^{m-1}\log\log n}
  \end{equation*}
  respectively.
\end{thm}

\begin{thm}\label{thm:digits}
  Let $b\in\{0,1,\dots,d-1\}$. The number of occurrences of the digit $b$ in a random
  multi-base representation of $n$ of the form~\eqref{eq:def-multi-base-repr}
  asymptotically follows a Gaussian distribution with mean and variance equal
  to
  \begin{equation*}
    \mu_n = \frac{\kappa}{d \log d}
    (\log n)^{m}+\fOh{(\log n)^{m-1}\log\log n} 
  \end{equation*}
  and
  \begin{equation*}
    \sigma_n^2 = \frac{\kappa(d-1)}{d^2 \log d} (\log n)^{m}
    +\fOh{(\log n)^{m-1}\log\log n}
  \end{equation*}
  respectively.
\end{thm}

The proofs of all these theorems are based on a saddle-point analysis of the
associated generating functions. As it turns out, the tail estimates are most
challenging, especially in the case $m=2$ (see Section~\ref{sec:tails}
for details). For the asymptotic analysis
of the various harmonic sums that occur, we apply the classical Mellin
transform technique, see
\cite{Flajolet-Gourdon-Dumas:1995:mellin}.


\section{The Generating Function}
\label{sec:generating-function}


We start with a generating function for our problem. As mentioned earlier, we define the set 
\begin{equation*}
  \Se=\lbrace p_1^{\alpha_1}p_2^{\alpha_2}\cdots p_m^{\alpha_m} 
  : \alpha_j \in \mathbb{N}\cup \{0\} \rbrace,
\end{equation*}
which is exactly the monoid that is freely generated by $p_1$, $p_2$, \dots,
$p_m$.  Note that the representations of $n$ correspond exactly to partitions
of $n$ into elements of $\Se$ where each term has multiplicity at most
$d-1$. The generating function for such partitions, where the first
variable~$z$ marks the size~$n$ and the second variable~$u$ marks the sum of
digits, can be written as
\begin{equation}\label{GF}
  F(z,u) = \prod_{h\in \Se} 
  \left( 1 + uz^h + u^2 z^{2h} + \dots + u^{d-1} z^{(d-1)h} \right) 
  = \prod_{h\in \Se}\frac{1-(u z^h)^d}{1-u z^h}.
\end{equation}

Likewise, we have the generating function
\begin{equation}\label{GF2}
  G(z,u)= \prod_{h\in \Se} 
  \left( 1 + uz^h + uz^{2h} + \dots + uz^{(d-1)h} \right) 
  = \prod_{h\in \Se} \left(1 + uz^h \frac{1-z^{(d-1)h}}{1-z^h} \right),
\end{equation}
where the second variable marks the Hamming weight (number of non-zero digits,
or equivalently number of distinct parts in a partition). For a digit
$b\in\{0,1,\dots,d-1\}$, whose occurrences will be marked by $u$, we use the generating
function
\begin{equation}\label{GF3}
  H_b(z,u)= \prod_{h\in \Se} 
  \left( 1 + z^h + \dots + uz^{bh} + \cdots + z^{(d-1)h} \right) 
  = \prod_{h\in \Se}\left(\frac{1-z^{dh}}{1-z^h} + (u-1) z^{bh}\right).
\end{equation}

Obviously, $F(z,1) = G(z,1) = H_b(z,1)$. We would like to apply the saddle-point
method to these generating functions. The trickiest part in this regard are the
rather technical tail estimates, especially when $m=2$, which will be discussed
in the next section. We will also need an asymptotic expansion in the central
region. To this end, we define the three functions
\begin{align*}
  f(t,u) = \log F(e^{-t},u) 
  &= \sum_{h \in \Se} 
  \flog[big]{1 + ue^{-ht} + u^2e^{-2ht} + \dots + u^{d-1}e^{-(d-1)ht}}, \\
  g(t,u) = \log G(e^{-t},u)
  &= \sum_{h \in \Se}
  \flog[big]{1 + ue^{-ht} + ue^{-2ht} + \dots + ue^{-(d-1)ht}}
  \intertext{and}
  h_b(t,u) = \log H_b(e^{-t},u)
  &= \sum_{h \in \Se}
  \flog[big]{1 + e^{-ht} + \dots + ue^{-bht} + \dots + e^{-(d-1)ht}}.
\end{align*}

\begin{lem}\label{lem:asy_central}
Suppose that $u$ lies in a fixed bounded positive interval around $1$, e.g. $u \in [1/2,2]$.
\begin{enumerate}
\item
For certain (real-)analytic functions $f_1(u)$, $f_2(u)$, \dots, $f_m(u)$ with 
\begin{equation*}
  f_m(u) = \log (1+u+\cdots+u^{d-1}) \prod_{j=1}^m\frac{1}{\log p_j},
\end{equation*}
we have the following asymptotic formula as $t \to 0^+$ ($t$ positive and
real), uniformly in $u$:
\begin{equation*}
  f(t,u) = \frac{f_m(u)}{m!} (\log 1/t)^m 
  + \frac{f_{m-1}(u)}{(m-1)!} (\log 1/t)^{m-1} + \dots 
  + f_1(u) (\log 1/t) + \fOh{1}.
\end{equation*}
Moreover,
\begin{equation*}
  \frac{\partial}{\partial t} f(t,u) = 
  - \frac{f_m(u)}{(m-1)!\, t} (\log 1/t)^{m-1} + \fOh{t^{-1} (\log 1/t)^{m-2} }
\end{equation*}
and
\begin{equation*}
  \frac{\partial^2}{\partial t^2} f(t,u) 
  = \frac{f_m(u)}{(m-1)!\, t^2} (\log 1/t)^{m-1} + \fOh{t^{-2} (\log 1/t)^{m-2} }.
\end{equation*}
Finally, there exists an $\eta > 0$ such that for complex $t$ with $\abs{\Im t} \leq \eta$, we have
\begin{equation*}
  \frac{\partial^3}{\partial t^3} f(t,u) = \fOh{(\Re t)^{-3} (\log 1/(\Re t))^{m-1} }
\end{equation*}
as $\Re t \to 0^+$, again uniformly in $u$.

\item Likewise, there exist functions $g_1(u)$, $g_2(u)$, \dots, $g_m(u)$ such
  that
  \begin{equation*}
    g(t,u) = \frac{g_m(u)}{m!} (\log 1/t)^m 
    + \frac{g_{m-1}(u)}{(m-1)!} (\log 1/t)^{m-1} + \dots 
    + g_1(u) (\log 1/t) + \fOh{1},
  \end{equation*}
  and the same conditions as in (1) hold with
  \begin{equation*}
    g_m(u) = \log (1+(d-1)u) \prod_{j=1}^m\frac{1}{\log p_j}.
  \end{equation*}

\item Moreover, for each digit $b\in\{0,1,\dots,d-1\}$, there exist functions
  $h_{b,1}(u)$, $h_{b,2}(u)$, \dots, $h_{b,m}(u)$ such that
  \begin{equation*}
    h_b(t,u) = \frac{h_{b,m}(u)}{m!} (\log 1/t)^m 
    + \frac{h_{b,{m-1}}(u)}{(m-1)!} (\log 1/t)^{m-1} + \dots 
    + h_{b,1}(u) (\log 1/t) + \fOh{1},
  \end{equation*}
  and the same conditions as in (1) hold with
  \begin{equation*}
    h_{b,m}(u) = \log (d-1+u) \prod_{j=1}^m\frac{1}{\log p_j}.
  \end{equation*}
\end{enumerate}
\end{lem}
 
\begin{proof}
  To prove the first part, we apply the
  classical Mellin transform technique to deal with the harmonic sums, see the
  paper of Flajolet, Gourdon and
  Dumas~\cite{Flajolet-Gourdon-Dumas:1995:mellin}. Consider first the
  Mellin transform
  \begin{equation*}
    Y(s,u) = \int_0^{\infty} \flog[big]{1 + ue^{-t} + u^2e^{-2t} 
      + \cdots + u^{d-1}e^{-(d-1)t}} t^{s-1}\,dt.
  \end{equation*}
  Integration by parts allows us to provide a meromorphic continuation
  (cf.\ Hwang~\cite{Hwang:2001:limit_theorems_number}). We have
  \begin{equation*}
    Y(s,u) = \frac{1}{s} \int_0^{\infty} t^s \frac{ue^{-t}+2u^2 e^{-2t} 
      + \dots + (d-1)u^{d-1}e^{-(d-1)t}}{1+ue^{-t} + \dots + u^{d-1}e^{-(d-1)t}}\,dt,
  \end{equation*}
  which exhibits the pole at $0$ with residue $\flog{1+u+u^2+\cdots+u^{d-1}}$,
  i.e., we have
  \begin{equation*}
    Y(s,u) \sim s^{-1}\log (1+u+\cdots+u^{d-1})
  \end{equation*}
  as $s \to 0$. By repeating this process one obtains a meromorphic
  continuation with further poles at $-1,-2,\ldots$.

  Moreover, since the integrand in the definition of $Y(s,u)$ decays
  exponentially as $\Re t \to \infty$, we can change the path of integration to
  the ray consisting of all complex numbers $t$ with $\Arg t = \epsilon > 0$,
  where $\epsilon$ is chosen small enough so that there is no $t$ with $\Arg t
  \leq \epsilon$ for which the expression inside the logarithm vanishes (this
  is possible since $u$ was assumed to be positive, so there are no real values
  of $t$ for which this happens). Set $\beta = e^{i \epsilon}$, and perform the
  change of variables $t = \beta v$ to obtain
  \begin{equation*}
    Y(s,u) = \beta^s \int_0^{\infty} 
    \flog[big]{1 + ue^{-\beta v} + u^2e^{-2\beta v} + \cdots +
      u^{d-1}e^{-(d-1)\beta v}}
    v^{s-1}\,dv.
  \end{equation*}
  If now $s = \sigma + i \tau$ with $\sigma > 0$, then the integral is
  uniformly bounded in $\tau$ for fixed $\sigma$, while the factor $\beta^s =
  e^{i \epsilon \sigma - \epsilon \tau}$ decays exponentially as $\tau \to
  \infty$. The same can be done for $\sigma = 0$ and for negative values of
  $\sigma$ (after suitable integration by parts) as well as negative $\tau$ (by
  symmetry). Therefore, we have
  \begin{equation*}
    Y(\sigma+ i\tau,u) = \fOh[big]{e^{-\epsilon\abs{\tau}}}
  \end{equation*}
  as $\tau \to \infty$, uniformly in $u$.

  Second, let us consider the Dirichlet series associated with the set $\Se$,
  i.e., $D(s) = \sum_{h \in \Se} h^{-s}$. It can be written as a product of
  elementary functions
  \begin{equation}\label{eq:mellin-sum}
    D(s) = \sum_{h \in \Se} h^{-s} = \prod_{j=1}^m\frac{1}{1-p_j^{-s}}.
  \end{equation}
  Each factor $1/(1-p_j^{-s})$ has a simple pole at $0$ and its singular
  expansion there is given by $1/(1-p_j^{-s}) \sim 1/(s \log p_j)$ as $s\to0$.

  Next we consider the Mellin transform of $f(t,u)$, which is given by $\f{Y}{s,u}
  \f{D}{s}$. This function has a pole of order $m+1$ at $s=0$, so the
  Laurent series of $Y(s,u)D(s)$ has the form
  \begin{equation*}
    \frac{f_m(u)}{s^{m+1}} + \frac{f_{m-1}(u)}{s^m} + \cdots 
    + \frac{f_1(u)}{s^2}+ \frac{f_0(u)}{s} + \cdots,
  \end{equation*}
  with $f_m(u)$ as indicated in the statement of our lemma. The other
  coefficients can be expressed in terms of certain improper
  integrals. Applying the Mellin inversion formula, we get
  \begin{equation*}
    f(t,u) = \frac{1}{2\pi i} \int_{c-i \infty}^{c+i \infty} Y(s,u)D(s) t^{-s}\,ds
  \end{equation*}
  for any $c > 0$. Following Flajolet, Gourdon and Dumas
  \cite[Theorem 4]{Flajolet-Gourdon-Dumas:1995:mellin}, we shift the line of integration to the left and pick up residues at the poles. This is
  possible because of the aforementioned growth properties of $Y(s,u)$. The
  main contribution comes from the pole at $s=0$, where the residue is indeed
  \begin{equation*}
    \frac{f_m(u)}{m!} (\log 1/t)^m + \frac{f_{m-1}(u)}{(m-1)!} (\log 1/t)^{m-1} 
    + \cdots + f_1(u) (\log 1/t)  + f_0(u).
  \end{equation*}
  There are further poles at all multiples of $2\pi i /\log p_j$ ($1 \leq j
  \leq m$), which are all simple poles (no two of them coincide) in view of the
  fact that the $p_j$ were assumed to be pairwise coprime, hence they only
  contribute $\fOh{1}$. In fact, the $\fOh{1}$ term can be replaced by a sum of
  $m$ Fourier series with periods $\log p_j$ ($1 \leq j \leq m$) that are given by
\begin{align*}
\Psi_j(t) &= \sum_{k \in \Z \setminus \{0\}} \operatornamewithlimits{Res}_{s = 2\pi i k /\log p_j} Y(s,u)D(s) t^{-s} \\
&= 
\sum_{k \in \Z \setminus \{0\}} Y \Big( \frac{2\pi i k}{\log p_j},u\Big) \cdot \frac{1}{\log p_j} \prod_{\substack{r=1 \\ r \neq j}}^m \frac{1}{1 - p_r^{-2\pi i k /\log p_j}} \exp \Bigl( -\frac{2\pi i k \log t}{\log p_j} \Bigr).
\end{align*}
 We remark that these Fourier series have exponentially decaying coefficients, since
  $Y(s,u)$ decays exponentially in imaginary direction, while Baker's theorem
  on linear forms in logarithms (see Chapter 12 of
  \cite{Cohen:2007:number-theory}) guarantees that
  \begin{equation*}
    \prod_{\substack{r=1 \\ r \neq j}}^m \frac{1}{\abs[big]{1 - p_r^{-2\pi i k /\log p_j}}}
  \end{equation*}
  is bounded above by a power of $k$: indeed, there exist constants $A_{\Lambda}, B_{\Lambda}$ (depending on the bases $p_1,p_2,\ldots,p_m$) such that
  \begin{equation*}
    \Lambda = \abs{k \log p_r - \ell \log p_j} \geq A_{\Lambda}k^{-B_{\Lambda}}
  \end{equation*}
for all $r \neq j$ and integers $k$ and $\ell$ not equal to $0$. Thus, if $\|\cdot\|$ denotes the distance to the nearest integer,
\begin{equation*}
  \left\|k \frac{\log p_r}{\log p_j} \right\| \geq \frac{A_{\Lambda}}{\log p_j}k^{-B_{\Lambda}}
\end{equation*}
and consequently
\begin{equation*}
  \abs[big]{1 - p_r^{-2\pi i k /\log p_j}}
  = \abs[Big]{1 - \exp \Bigl( - \frac{2\pi i k \log p_r}{\log p_j} \Bigr)}
  \geq 4 \left\| \frac{k \log p_r}{\log p_j} \right\| 
  \geq \frac{4A_{\Lambda}}{\log p_j}k^{-B_{\Lambda}}.
\end{equation*}
It follows that
  \begin{equation*}
    \prod_{\substack{r=1 \\ r \neq j}}^m \frac{1}{\abs[big]{1 - p_r^{-2\pi i k /\log p_j}}} = \fOh{k^{(m-1)B_{\Lambda}}},
  \end{equation*}
which in turn means that
\begin{equation*}
  \operatornamewithlimits{Res}_{s = 2\pi i k /\log p_j} Y(s,u)D(s) t^{-s}
  = \fOh{ \abs{k}^{(m-1)B_{\Lambda}} e^{-2\pi\epsilon\abs{k}/\log p_j}}.
\end{equation*}
Thus each of the Fourier series $\Psi_j$ is convergent and indeed represents a smooth function. This proves the asymptotic formula for
  $f(t,u)$. The derivatives $\frac{\partial}{\partial t} f(t,u)$ and
  $\frac{\partial^2}{\partial t^2} f(t,u)$ have Mellin transforms
  $(1-s)Y(s-1,u)D(s-1)$ and $(s-1)(s-2)Y(s-2,u)D(s-2)$ respectively, so
  essentially the same arguments apply, now with the main terms coming from the
  poles at $1$ and $2$ respectively.

  It remains to prove the estimate for the third derivative. Note that it can
  be written as
  \begin{equation*}
    \frac{\partial^3}{\partial t^3} f(t,u) = \sum_{h \in \Se} h^3 e^{-ht} 
    \frac{Q(e^{-ht},u)}{(1+ue^{-ht} + \cdots + u^{d-1}e^{-(d-1)ht})^3},
  \end{equation*}
  where $Q$ is some polynomial. If we choose $\eta$ (recall that our result
  will be valid for $\abs{\Im t} \leq \eta$) small enough so that the
  denominator stays away from $0$ (compare the analysis of $Y(s,u)$ above), the
  last factor is uniformly bounded by a constant. The Mellin transform of
  \begin{equation*}
    \sum_{h \in \Se} h^3 e^{-ht}
  \end{equation*}
  is given by $\Gamma(s)D(s-3)$, to which we can apply the same arguments as
  for the harmonic sums encountered before. The dominant singularity is clearly
  a pole of order $m$ at $s=3$ in this case, so that the desired estimate
  follows immediately.

  The proofs of the second and the third part of Lemma~\ref{lem:asy_central}
  are analogous.
\end{proof}

\section{Estimating the Tails}\label{sec:tails}

For our application of the saddle-point method, we need to estimate the tails
(i.e., the parts where $z$ is away from the positive real axis) of the
generating functions given in~\eqref{GF},~\eqref{GF2} and~\eqref{GF3}. This is done in the
following sequence of lemmas. First of all, let us introduce some notation. For
$r > 0$, we set
\begin{equation*}
  \Se(r) = \Se \cap [1,1/r] = \{h \in \Se\,:\, hr \leq 1\}.
\end{equation*}
It is straightforward to prove that
\begin{equation}\label{eq:card_est}
\abs{\Se(r)} = \frac{(\log 1/r)^m}{m! \prod_{j=1}^m \log p_j} + \fOh{(\log 1/r)^{m-1}}
\end{equation}
as $r\to0^+$. Note that later (starting with the next section), $r$ will be determined by the saddle point equation.

\begin{lem}\label{lem:tail1}
  Let $u$ be in the interval $[\frac12,2]$, and let $z = e^{-r+2\pi iy}$ with
  $r > 0$ and $y \in [-\frac12,\frac12]$. There exists an absolute constant $C$
  such that
  \begin{align*}
    \frac{\abs{F(z,u)}}{F(\abs{z},u)}
    &\leq \fexp[bigg]{- C \sum_{h \in \Se(r)} \dni{hy}^2 }, \\
    \frac{\abs{G(z,u)}}{G(\abs{z},u)}
    &\leq \fexp[bigg]{- C \sum_{h \in \Se(r)} \dni{hy}^2 }
    \intertext{and}
    \frac{\abs{H_b(z,u)}}{H_b(\abs{z},u)}
    &\leq \fexp[bigg]{- C \sum_{h \in \Se(r)} \dni{hy}^2 },
  \end{align*}
  where $\dni{\,\cdot\,}$ denotes the distance to the nearest integer.
\end{lem}

\begin{proof}
  For positive real $a$ and complex $w$, we have the two identities
  \begin{equation*}
    \frac{\abs{1+a w}^2}{(1+a \abs{w})^2} = 1 - \frac{2a(\abs{w}-\Re w)}{(1+a\abs{w})^2}
  \end{equation*}
  and
  \begin{equation*}
    \frac{\abs{1+a w+a w^2}^2}{(1+a \abs{w}+a \abs{w}^2)^2}
    = 1 - \frac{2a(\abs{w}-\Re w)(1+2\abs{w}+a\abs{w}^2+2\Re w)}{(1+a\abs{w}+a\abs{w}^2)^2}.
  \end{equation*}
  Assuming that $a \in [\frac12,2]$ and $\abs{w} \leq 2$, we get
  \begin{equation}\label{eq:est1}
    \frac{\abs{1+a w}^2}{(1+a \abs{w})^2} \leq 1 - \frac{1}{25}(\abs{w}-\Re w)
    \leq \fexp[Big]{- \frac{1}{25} (\abs{w} - \Re w) }
  \end{equation}
  and
  \begin{equation}\label{eq:est2}
    \frac{\abs{1+a w+a w^2}^2}{(1+a \abs{w}+a \abs{w}^2)^2} 
    \leq 1 - \frac{1}{169} (\abs{w}-\Re w) 
    \leq \fexp[Big]{- \frac{1}{169} (\abs{w} - \Re w) }.
  \end{equation}
  Now let $d$ be even, set $a = u$ and $w = z^h$, so that~\eqref{eq:est1},
  together with the triangle inequality, yields
  \begin{align*}
    &\abs{1 + u z^h + u^2 z^{2h} + \cdots + u^{d-1}z^{(d-1)h}} \\
    &\qquad\leq \abs{1+u z^h} + u^2 \abs{z}^{2h} \abs{1+u z^h} 
    + \cdots + u^{d-2}\abs{z}^{(d-2)h} \abs{1+u z^h} \\
    &\qquad\leq \big( 1 + u \abs{z}^h + u^2 \abs{z}^{2h} + \cdots 
      + u^{d-1}\abs{z}^{(d-1)h}\big) 
    \fexp[Big]{-\frac{1}{50} \big(\abs{z}^h - \f{\Re}{z^h} \big)}.
  \end{align*}
  Taking the product over all $h \in \Se$ gives
  \begin{align*}
    \abs{F(z,u)} &\leq F(\abs{z},u) \fexp[Big]{- \frac{1}{50} 
      \sum_{h \in \Se} \big(\abs{z}^h - \f{\Re}{z^h} \big)} \\
    &= F(\abs{z},u) \fexp[Big]{- \frac{1}{50}
      \sum_{h \in \Se} e^{-hr}\left( 1- \cos(2\pi hy) \right)} \\
    &\leq F(\abs{z},u) \fexp[Big]{- \frac{1}{50e} 
      \sum_{h \in \Se(r)} \left( 1- \cos(2\pi hy) \right)} \\
    &\leq F(\abs{z},u)
    \fexp[Big]{- \frac{8}{50e} \sum_{h \in \Se(r)} \dni{hy}^2 },
  \end{align*}
  which proves the first statement of the lemma with $C = 4/(25e)$. For odd
  $d$, we can argue in a similar fashion, but we also apply~\eqref{eq:est2}
  (with $a = 1$ and $w = uz^h$) and use the triangle inequality in the
  following way:
  \begin{multline*}
    \abs[big]{1 + u z^h + u^2 w^2 + \cdots + u^{d-1}z^{(d-1)h}} \\
    \leq \abs[big]{1+u z^h+u^2 z^{2h}} + u^3 \abs{z}^{3h} \abs[big]{1+u z^h} 
    + \cdots + u^{d-2}\abs{z}^{(d-2)h} \abs[big]{1+u z^h}.
  \end{multline*}
  For the generating function $G(z,u)$, the reasoning is fully analogous, but
  we also have to use~\eqref{eq:est2} with $a=u$ and $w=z^h$. A similar
  situation occurs for $H_b(z,u)$.
\end{proof}

Next we estimate the sum that occurs in the previous lemma. When $m > 2$,
relatively simple estimates suffice for our purposes, while we need an
additional auxiliary result in the case that $m=2$. The following lemma
provides the necessary estimates.

\begin{lem}\label{lem:tail2}
Let $r > 0$ and $y \in [-\frac12,\frac12]$, and set
\begin{equation*}
  \Sigma = \Sigma(r,y) =  \sum_{h \in \Se(r)} \dni{hy}^2,
\end{equation*}
where again $\dni{\,\cdot\,}$ denotes the distance to the nearest integer.
For sufficiently small $r$, we have the following estimates for $\Sigma$.
\begin{enumerate}
\item[(a)] If $\abs{y} \leq r/2$, then $\Sigma \geq A_1 (y/r)^2 (\log
  (1/r))^{m-1}$ for some positive constant $A_1$ (that only depends on $m$ and
  the set of bases $\{p_1,p_2,\dots,p_m\}$).
\item[(b)] If $\abs{y} \geq r/2$, then $\Sigma \geq A_2 (\log (1/r))^{m-1}$ for
  some positive constant $A_2$ (that also only depends on $m$ and the set of
  bases $\{p_1,p_2,\dots,p_m\}$).
\end{enumerate}
Now let $m = 2$. For any constant $K > 0$ and any $\delta > 0$, there exists a constant $B > 0$ depending on $p_1,p_2,K$ and $\delta$ such that
the following holds for sufficiently small $r$.
\begin{enumerate}
\item[(c)] We have $\Sigma \geq K \log (1/r)$, except when $y$ lies in a
  certain set $E(K,r)$ of Lebesgue measure at most $Br^{1-\delta}$.
\end{enumerate}
\end{lem}

\begin{proof}
For better readability, the proof is split into several claims.

  \begin{beh}\label{beh:a-correct}
    Statement (a) is correct.
  \end{beh}
  \begin{proof}[Proof of \behref{beh:a-correct}]
    Let $\abs{y} \leq r/2$, which implies $\abs{hy} \leq \frac12$ for all $h
    \in \Se(r)$. Then we have
    \begin{equation*}
      \Sigma = \sum_{h \in \Se(r)} \dni{hy}^2 
      = \sum_{h \in \Se(r)} h^2y^2
      \geq  \sum_{\substack{h \in \Se(r) \\ h \notin \Se(r/\rho)}} h^2y^2
      \geq \rho^2 (y/r)^2 \left( \abs{\Se(r)} - \abs{\Se(r/\rho)} \right)
    \end{equation*}
    for any $\rho > 0$. If we take $\rho$ sufficiently small and apply the
    asymptotic formula in~\eqref{eq:card_est}, we obtain estimate (a).
  \end{proof}

  \begin{beh}\label{beh:b-correct-1}
    $A_2$ can be chosen in such a way that statement (b) holds for $\abs{y} \leq r^{2/3}$.
  \end{beh}
  \begin{proof}[Proof of \behref{beh:b-correct-1}]
    Let us assume that $r/2 \leq \abs{y} \leq r^{2/3}$. Then we have $\log
    \abs{1/y} \geq \frac23 \log (1/r)$, and essentially the same idea as above
    works again. We obtain
    \begin{equation*}
      \Sigma = \sum_{h \in \Se(r)} \dni{hy}^2 \geq \sum_{h \in \Se(2\abs{y})} h^2y^2
      \geq  \sum_{\substack{h \in \Se(2\abs{y}) \\ h \notin \Se(\abs{y}/\rho)}} h^2y^2
      \geq \rho^2 \left( \abs{\Se(2\abs{y})} - \abs{\Se(\abs{y}/\rho)} \right),
    \end{equation*}
    and formula~\eqref{eq:card_est} can be applied again to obtain (b).
  \end{proof}
  
  We are left with the case that $\abs{y} > r^{2/3}$, so we will assume this
  from now on. By Dirichlet's approximation theorem, there exists a rational
  number $a/q$ (with coprime $a$ and $q$) such that $q \leq r^{-2/3}$ and
  \begin{equation*}
    \abs[bigg]{y - \frac{a}{q}} \leq \frac{r^{2/3}}{q}.
  \end{equation*}

  \begin{beh}\label{beh:many-elements-h1}
    There exists a positive constant $c_1$ that only depends on $m$ and the set
    of bases $\{p_1,p_2,\ldots,p_m\}$ such that for small enough $r$ and any
    coprime integers $a$, $q$ with $q \leq r^{-2/3}$, there are at least
    \begin{equation*}
      c_1 (\log q)(\log 1/r)^{m-1}
    \end{equation*}
    many elements $h_1 \in \Se(r^{1/3})$ with $\dni{ah_1/q} \geq 1/q$.
  \end{beh}

  \begin{proof}[Proof of \behref{beh:many-elements-h1}]
    For $q = 1$, the statement is trivial, so we assume that $q \neq 1$. Let us now distinguish whether $q$ is in the set $\Se$ or
    not.

    If $q \in \Se$, then write $q = p_1^{\alpha_1}p_2^{\alpha_2} \cdots
    p_m^{\alpha_m}$. We have
    \begin{equation*}
      A = \max(\alpha_1,\alpha_2,\ldots,\alpha_m) 
      \geq \frac{\log q}{\log(p_1p_2\dots p_m)}.
    \end{equation*}
    Suppose that $\alpha_i = A$. Consider the elements $h_1 =
    p_1^{\beta_1}p_2^{\beta_2} \cdots p_m^{\beta_m} \in \Se$ with $0 \leq
    \beta_i < \alpha_i = A$. For any of these $h_1$, the number $ah_1/q$ is not
    an integer and thus $\dni{ah_1/q} \geq 1/q$. Let us now find a lower bound
    for the number of such elements~$h_1$. Using~\eqref{eq:card_est} (applied
    to the set $\Se_i = \{s \in \Se\,:\, p_i \nmid s\}$), we find that
    for some positive constants $\hat c_1$ and $c_1$, there exist at least
    \begin{equation*}
      \hat c_1 A \abs[big]{\Se_i(r^{1/3})}
      \geq c_1 (\log q)(\log 1/r)^{m-1}
    \end{equation*}
    elements $h_1 \in \Se$ with $h_1 \leq r^{-1/3}$.

    If $q \notin \Se$, then we clearly have $\dni{h_1a/q} \geq 1/q$ for all
    $h_1 \in \Se$, so the same statement as in the first case holds again.
  \end{proof}

  \begin{beh}\label{beh:many-elements-h}
    There exists a positive constant $c$ that only depends on $m$ and the set of bases $\{p_1,p_2,\ldots,p_m\}$ such that for sufficiently small $r$ and $r^{2/3} < \abs{y} \leq \frac12$, there are at least
    \begin{equation*}
      c (\log 1/r)^{m-1}
    \end{equation*}
    many elements $h \in \Se(r)$ with $\dni{hy} \geq 1/(3p_1)$.
  \end{beh}

  \begin{proof}[Proof of \behref{beh:many-elements-h}]
    Let us divide the interval $[1/q,1/2]$ into subintervals
    \begin{equation*}
      \text{$I_0 = [1/(2p_1),1/2]$, $I_1 = [1/(2p_1^2),1/(2p_1)]$, \dots}
    \end{equation*}
    whose ends have a ratio of $p_1$ (except possibly for the last one). There
    are at most
    \begin{equation*}
      \log(q/2)/\log(p_1) \leq c_2\log q
    \end{equation*}
    such intervals.

    By \behref{beh:many-elements-h1} and the pigeonhole principle, we can
    choose one of these intervals ($I_j$, say) such that for at least
    $c_1/c_2\, (\log 1/r)^{m-1}$ distinct numbers $h_1 \in \Se$ with $h_1 \leq
    r^{-1/3}$, the number $\dni{h_1a/q}$ lies in this interval~$I_j$, i.e., we
    have $1/(2p_1^{j+1}) \leq \dni{h_1a/q} \leq 1/(2p_1^j)$.

    Now we have
    \begin{equation*}
      \dni[bigg]{\frac{h_1 p_1^j a}{q}} 
      = p_1^j \dni[bigg]{\frac{h_1 a}{q}} \geq \frac{1}{2p_1},
    \end{equation*}
    which means that we have at least $c_1/c_2\, (\log 1/r)^{m-1}$ elements $h
    = h_1 p_1^j \in \Se$ with $\|ah/q\| \geq 1/(2p_1)$ and
    \begin{equation*}
      h = h_1 p_1^j \leq h_1 q \leq r^{-1/3} r^{-2/3} = \frac1r.
    \end{equation*}
    All of these numbers $h$ are therefore in the set $\Se(r)$. For sufficiently small $r$, it follows that
    \begin{equation*}
      \dni{ hy} \geq \dni[bigg]{\frac{h a}{q}} - \frac{r^{2/3}h}{q} \geq
      \frac{1}{2p_1} - \frac{r^{2/3}h_1 q}{q} \geq \frac{1}{2p_1} - r^{1/3} \geq
      \frac{1}{3p_1},
    \end{equation*}
    which proves the claim.
  \end{proof}

  \begin{beh}\label{beh:b-correct-2}
    $A_2$ can be chosen in such a way that statement (b) holds for $\abs{y} \geq r^{2/3}$.
  \end{beh}
  \begin{proof}[Proof of \behref{beh:b-correct-2}]
    The result follows from \behref{beh:many-elements-h} since
    \begin{equation*}
      \Sigma \geq c \Big( \log \frac1r \Big)^{m-1}
      \cdot \Big( \frac{1}{3p_1} \Big)^2 = A_2 \Big(\! \log \frac1r \Big)^{m-1}
    \end{equation*}
    for $A_2 = c/(9p_1^2)$ if $r$ is sufficiently small.
  \end{proof}

So (b) is now proven in both cases, and it remains to prove statement (c) of the lemma, so assume that $m=2$. Choose
  some $\epsilon \in (0,\delta)$, set
  \begin{equation*}
    L = \lfloor (1-\epsilon)\log_{p_1} 1/r \rfloor
  \end{equation*}
  and define, for a positive integer $M$, the set
  \begin{equation*}
    D(M) = \big\{v \in [0,1]\,:\, \dni[big]{p_1^\ell v } < p_1^{-2}
    \text{ for $0 \leq \ell \leq L$ with at most $M$ exceptions}\big\}.
  \end{equation*}
  The constant $M$ will be chosen appropriately at the end of the proof.
  
  We get the following result, which almost proves (c).

  \begin{beh}\label{beh:almost-c}
    Set $R = \lfloor \epsilon \log_{p_2} 1/r \rfloor$. If $y$ is not contained
    in the set
    \begin{equation*}
      E = \bigcup_{k \leq R} \{y \in [-\tfrac12,\tfrac12] \,:\, p_2^k y \bmod 1 \in D(M) \},
    \end{equation*}
    then
    \begin{equation*}
      \Sigma \geq \epsilon p_1^{-2} M \log_{p_2} 1/r.
    \end{equation*}
  \end{beh}

  \begin{proof}[Proof of \behref{beh:almost-c}]
    By our assumptions, there is no $k \leq R$ such that $p_2^k y \bmod 1 \in
    D(M)$. Therefore, for a fixed $k$ the inequality $\dni[big]{p_1^\ell p_2^k
      y} \geq p_1^{-2}$ holds for more than $M$ choices of $\ell \leq L$.
    Moreover, we have $p_1^\ell p_2^k \leq r^{-1+\epsilon} \cdot r^{-\epsilon}
    = r^{-1}$ for all  such $k$ and $\ell$.

    It follows that
    \begin{equation*}
      \Sigma = \sum_{h \in \Se(r)} \dni{hy}^2 
      \geq \sum_{\ell \leq L} \sum_{k \leq R} 
      \dni[big]{p_1^\ell p_2^k y}^2 \geq (R+1) M p_1^{-2} 
      \geq \epsilon p_1^{-2} M \log_{p_2} 1/r,
    \end{equation*}
    which is what we wanted to show.
  \end{proof}

  It remains to show that the set $E$ is small. This is done in the following
  two claims.

  \begin{beh}\label{beh:meas-DM}
    The Lebesgue measure of the set $D(M)$ is at most $\fOh{L^M p_1^{M-L} }$.
  \end{beh}

  \begin{proof}[Proof of \behref{beh:meas-DM}]
    First, note that $\dni[big]{p_1^\ell v} \geq p_1^{-2}$ unless the
    $(\ell+1)$-th and the $(\ell+2)$-th digit after the decimal\footnote{We
      should rather correctly say ``$p_1$-point'' instead of ``decimal point''
      since $p_1$ is the base of our numeral system, but this may lead to even
      more confusion.} point in the $p_1$-adic expansion of $v$ are either both
    $0$ or both $p_1-1$. For an upper bound, we relax this condition to both
    digits being equal.

    Therefore, for an element of $D(M)$, at least $L-M+1$ of the first $L+2$
    digits have to be equal to the previous digit. Allowing exactly $j \leq M$
    exceptions, there are $\binom{L+1}{j}$ number of ways to choose the
    ``exceptional'' digits. Moreover, each digit that has to be equal to the
    previous one reduces the Lebesgue measure by a factor of $p_1$.

    Putting everything together, we end up finding that the Lebesgue measure of $D(M)$ is at most
    \begin{equation*}
      \sum_{j=0}^M \binom{L+1}{j}p_1^{-(L+1)+j} = \fOh{L^M p_1^{M-L} },
    \end{equation*}
    which proves the claim.
  \end{proof}

  We need one more claim, which concerns the size of the exceptional set~$E$.

  \begin{beh}\label{beh:meas-E}
    The set $E$ has Lebesgue measure $\fOh{r^{1-\epsilon} (\log 1/r)^{M+1}}$.
  \end{beh}

  \begin{proof}[Proof of \behref{beh:meas-E}]
    Since $y \in [-\frac12,\frac12]$ (an interval of length~$1$) and
    $p_2^k$ is an integer, the Lebesgue measure $\lambda$ is preserved under
    taking the pre-image of $v\mapsto p_2^k v \bmod 1$. Therefore, we have
    \begin{equation*}
      \f{\lambda}{\{y \,:\, p_2^k y \bmod 1 \in D(M) \}}
      = \f{\lambda}{D(M)}
    \end{equation*}
    and obtain
    \begin{equation*}
      \lambda(E) \leq \sum_{k \leq R} \lambda(D(M)) 
      = \fOh{R L^M p_1^{M-L}} 
      = \fOh{r^{1-\epsilon} (\log 1/r)^{M+1}}.
    \end{equation*}
    Note that the implied constant only depends on $p_1$, $p_2$, $M$ and
    $\epsilon$.
  \end{proof}

  If we choose $M = \lceil K \epsilon^{-1} p_1^2\log p_2 \rceil$, then
  statement (c) follows from the claims above (in particular,
  \behref{beh:almost-c} and \behref{beh:meas-E}) with exceptional set $E =
  E(K,r)$. Note that $\lambda(E) = \fOh{r^{1-\epsilon} (\log 1/r)^{M+1}} =
  \fOh{r^{1-\delta}}$. This completes the proof.
\end{proof}

\section{Application of the Saddle-Point Method}\label{sec:saddle}

We are now ready to apply the saddle-point method (see Chapter VIII
of~\cite{Flajolet-Sedgewick:2009:analy} for an excellent introduction), which
gives us asymptotic formulas for the coefficients of the generating functions
$F(z,u)$, $G(z,u)$ and $H_b(z,u)$. In the following, we use the notations $f_t(t,u)$,
$f_{tt}(t,u)$, \ldots\@ for the derivatives of $f$ with respect to the first
coordinate.

\begin{lem}\label{lem:saddle}
  Let $u \in [\frac12,2]$, and define $r > 0$ implicitly by the saddle-point
  equation
  \begin{equation*}
    n = - f_t(r,u).
  \end{equation*}
  The coefficients of $F(z,u)$ satisfy the asymptotic formula
  \begin{equation*}
    [z^n] F(z,u) = \frac{1}{\sqrt{2\pi f_{tt}(r,u)}} e^{nr+f(r,u)} 
    \big( 1 + \fOh[big]{(\log n)^{-(m-1)/5}} \big),
  \end{equation*}
  uniformly in $u$. Likewise, if we define $r > 0$ by
  \begin{equation*}
    n = - g_t(r,u),
  \end{equation*}
  then the coefficients of $G(z,u)$ satisfy the asymptotic formula
  \begin{equation*}
    [z^n] G(z,u) = \frac{1}{\sqrt{2\pi g_{tt}(r,u)}} e^{nr+g(r,u)}
    \big( 1 + \fOh[big]{(\log n)^{-(m-1)/5}} \big),
  \end{equation*}
  uniformly in $u$, and if we define $r > 0$ by
  \begin{equation*}
    n = - h_{a,t}(r,u),
  \end{equation*}
  then the coefficients of $H_b(z,u)$ satisfy the asymptotic formula
  \begin{equation*}
    [z^n] H_b(z,u) = \frac{1}{\sqrt{2\pi h_{b,tt}(r,u)}} e^{nr+h_b(r,u)}
    \big( 1 + \fOh[big]{(\log n)^{-(m-1)/5}} \big).
  \end{equation*}
\end{lem}

Let us first give a short outline on the proof, which we only present for $F$, since the other two cases are analogous. We start by using Cauchy's
integral formula to extract the coefficient of $z^n$ from $F(z,u)$. After the subsequent change to polar coordinates ($z = e^{-(r+it)}$), we choose $r$ to satisfy the saddle point equation. Thus the Taylor expansion in the central
region simplifies (the first order term vanishes).
Lemma~\ref{lem:asy_central} shows that $r$ is of order $(\log
n)^{m-1}/n$. In the central region (to be defined later), the error term is
$\fOh{\log(1/r)^{-(m-1)/5}}$ by Lemma~\ref{lem:asy_central}, and we can
complete the tails to get a Gaussian integral. The remaining parts of the
integral are estimated by means of Lemmas~\ref{lem:tail1}
and~\ref{lem:tail2}. If $m > 2$, parts (a) and (b) of Lemma~\ref{lem:tail2}
already give sufficiently strong bounds. In the case that $m = 2$, we have to
divide the tails further into a small ``exceptional part'', where we apply (b),
and the rest, where the stronger bound from (c) holds.

So much for the overview; let us start with the actual proof now.

\begin{proof}
  By Cauchy's integral formula, we have
  \begin{equation*}
    [z^n]F(z,u)=\frac{1}{2\pi i}\oint_{\mathcal{C}}F(z,u)\frac{dz}{z^{n+1}},
  \end{equation*}
  where $\mathcal{C}$ is a circle around 0 with radius less than $1$. Let $r>0$
  and perform the change of variables $z = e^{-t} = e^{-(r+i\tau)}$, so that
  the integral becomes
  \begin{equation}\label{coef1}
    [z^n]F(z,u)=\frac{1}{2\pi}\int_{-\pi}^{\pi}\fexp{nr+f(r+i\tau,u)+in\tau}d\tau.
  \end{equation}
  Now we choose $r=r(n,u)>0$ to be the unique positive solution of the
  saddle-point equation
  \begin{equation}\label{eqr}
    n=-f_t(r,u).
  \end{equation}
  Let $c$ be a constant such that $(m-1)/3 < c < (m-1)/2$, we choose
  specifically $c = 2(m-1)/5$. Consider first the integral
  \begin{equation*}
    I_0=\frac{1}{2\pi}\int_{-r(\log 1/r)^{-c}}^{r(\log 1/r)^{-c}}
    \fexp{nr+f(r+i\tau,u)+in\tau}d\tau.
  \end{equation*}
  For $\abs{\tau}\leq r(\log 1/r)^{-c}$, using Taylor expansion and Lemma
  \ref{lem:asy_central}, we have
  \begin{align*}
    f(r+i\tau,u)
    & =f(r,u)+i f_{t}(r,u) \tau -f_{tt}(r,u)\frac{\tau^2}{2}
    +\fOh[Big]{\abs{\tau}^3\sup_{\abs{y} \leq \tau} \abs{f_{ttt}(r+iy,u)}}\\
    & =f(r,u)+if_{t}(r,u) \tau -f_{tt}(r,u)\frac{\tau^2}{2}
    +\fOh{(\log 1/r)^{m-1-3c}}.
  \end{align*}
  Therefore, by the definition of $r$ in \eqref{eqr}, we have 
  \begin{equation*}
    I_0=\frac{e^{nr+f(r,u)}}{2\pi}\int_{-r(\log 1/r)^{-c}}^{r(\log 1/r)^{-c}}
    \fexp{-f_{tt}(r,u)\frac{\tau^2}{2}}d\tau\, \big(
    1+\fOh{(\log 1/r)^{m-1-3c}}\big).
  \end{equation*}
  Furthermore,
  \begin{align*}
    &\int_{-r(\log 1/r)^{-c}}^{r(\log 1/r)^{-c}}
    \fexp{-f_{tt}(r,u)\frac{\tau^2}{2}}d\tau \\
    &\qquad= \int_{-\infty}^{\infty}\fexp{-f_{tt}(r,u)\frac{\tau^2}{2}}d\tau
    -2\int_{r(\log 1/r)^{-c}}^{\infty}\fexp{-f_{tt}(r,u)\frac{\tau^2}{2}}d\tau\\
    &\qquad= \sqrt{\frac{2\pi}{f_{tt}(r,u)}}
    -2\int_{r(\log 1/r)^{-c}}^{\infty}\fexp{-f_{tt}(r,u)\frac{\tau^2}{2}}d\tau,
  \end{align*}
  and 
  \begin{align*}
    0 \leq \int_{r(\log 1/r)^{-c}}^{\infty}\fexp{-f_{tt}(r,u)\frac{\tau^2}{2}}d\tau
    & \leq \int_{r(\log 1/r)^{-c}}^{\infty}
    \fexp{-\frac{\tau}{2}f_{tt}(r,u)r(\log 1/r)^{-c}}d\tau \\
    & = \frac{2\fexp{-f_{tt}(r,u)r^2(\log 1/r)^{-2c}/2}}{
      f_{tt}(r,u)r(\log 1/r)^{-c}}\\
    & = \fOh{r (\log 1/r)^{-(m-1-c)} e^{-\gamma (\log 1/r)^{m-1-2c}} }
  \end{align*}
  for a constant $\gamma > 0$. Since $m-1-2c = (m-1)/5>0$, the $\Oh$-term goes to   zero faster than any power of $\log 1/r$. Hence we have
  \begin{equation}
    I_0=\frac{e^{nr+f(r,u)}}{\sqrt{2\pi f_{tt}(r,u)}}
    \big(1+\fOh[big]{(\log 1/r)^{m-1-3c}}\big)
    =\frac{e^{nr+f(r,u)}}{\sqrt{2\pi f_{tt}(r,u)}}
    \big(1+\fOh[big]{(\log n)^{-(m-1)/5}}\big).
  \end{equation}

  It remains to show that the rest of the integral in \eqref{coef1} is small
  compared to $I_0$. To this end, note for comparison that $1/\sqrt{2\pi
    f_{tt}(r,u)}$ is of order $r (\log 1/r)^{-(m-1)/2}$. Now consider
  \begin{equation*}
    I_1=\int_{r(\log 1/r)^{-c}}^{\pi}\fexp{nr+f(r+i\tau,u)+in\tau}\,d\tau.
  \end{equation*}
  Then
  \begin{align*}
    \abs{ I_1} & \leq e^{nr+f(r,u)} \int_{r(\log 1/r)^{-c}}^{\pi}
    \fexp{\f{\Re}{f(r+i\tau,u)-f(r,u)}}\,d\tau \\
    &= e^{nr+f(r,u)} \int_{r(\log 1/r)^{-c}}^{\pi}
    \frac{\abs{F(e^{-(r+i\tau)},u)}}{F(e^{-r},u)}\,d\tau.
  \end{align*}

  If $m\geq 3$, then we can use Lemma~\ref{lem:tail1} and parts (a) and (b) of
  Lemma~\ref{lem:tail2} to show that the integrand $\abs{F(e^{-(r+i\tau)},u)}/F(e^{-r},u)$ on the right hand side is
  $\fOh{\fexp{-CA_1/(2\pi)^2\cdot (\log 1/r)^{m-1-2c}}}$ for $\abs{\tau} \leq
  \pi r$ and $\fOh{\fexp{-CA_2 (\log 1/r)^{m-1}}}$ otherwise, which
  immediately shows that
  \begin{equation*}
    \abs{I_1} = \fOh{e^{nr+f(r,u)} \Bigl( r \fexp{- CA_1/(2\pi)^2\cdot (\log 1/r)^{m-1-2c}} + \fexp{-CA_2 (\log 1/r)^{m-1}} \Bigr) }.
  \end{equation*}
  For $m=2$, we need to be more careful. Again, part (a) of
  Lemma~\ref{lem:tail2} can be used for the interval where $\abs{\tau} \leq \pi r$,
  with the same bound as above. The rest of the integral is split again: we
  choose a constant $K > 0$ such that $CK > 1$ ($C$ as in
  Lemma~\ref{lem:tail1}), and $\delta > 0$ such that $\delta < CA_2$ ($A_2$ as
  in Lemma~\ref{lem:tail2}).
  
  If $y = -\tau/(2\pi)$ is not in the exceptional set $E(K,r)$ as defined in
  Lemma~\ref{lem:tail2}, then we have
  \begin{equation*}
    \frac{\abs{F(e^{-(r+i\tau)},u)}}{F(e^{-r},u)} 
    = \fOh{\fexp{-CK \log 1/r}} 
    = \fOh{r^{CK} }.
  \end{equation*}
  By part (c) of Lemma~\ref{lem:tail2}, the set of $\tau$-values for which this
  estimate does not hold has Lebesgue measure $\fOh{r^{1-\delta}}$, and we have
  the estimate
  \begin{equation*}
    \frac{\abs{F(e^{-(r+i\tau)},u)}}{F(e^{-r},u)} 
    = \fOh{\fexp{-CA_2 \log 1/r}} 
    = \fOh{r^{CA_2} }
  \end{equation*}
  for all those $\tau$. Putting everything together shows that
  \begin{equation*}
    \abs{I_1} = \fOh[big]{e^{nr + f(r,u)} 
      \big( r \fexp[big]{- CA_1 (\log 1/r)^{1/5} }
        + r^{CK} + r^{CA_2+1-\delta} \big) },
  \end{equation*}
  which again means that $I_1$ is negligible, since the exponents $CK$ and
  $CA_2+1-\delta$ are both $> 1$. The same reasoning can of course be applied
  to
  \begin{equation*}
    I_2=\int_{-\pi}^{-r(\log 1/r)^{-c}}\fexp{nr+f(r+i\tau,u)+in\tau}\,d\tau.
  \end{equation*}
  This finishes the proof for the function $F(z,u)$. The proof for $G(z,u)$ and
  $H_b(z,u)$ is analogous.
\end{proof}

\section{The Number of Representations}

It is straightforward now to prove our main results.

\begin{proof}[Proof of Theorems~\ref{thm:asy-general} and~\ref{thm:asy-2}]
  We specialize by $u=1$ in Lemma~\ref{lem:saddle}, which gives us
  \begin{equation*}
    P(n) = [z^n] F(z,1) = \frac{1}{\sqrt{2\pi f_{tt}(r_0,1)}} e^{nr_0+f(r_0,1)}
    \big( 1 + \fOh[big]{(\log n)^{-(m-1)/5}} \big),
  \end{equation*}
  where $r_0$ is given by the saddle-point equation $n = -f_t(r_0,1)$ (as its unique
  positive solution). Making use of Lemma~\ref{lem:asy_central}, we get
  \begin{equation*}
    n = \frac{f_m(1)}{(m-1)!r_0} (\log 1/r_0)^{m-1} + \fOh{(\log 1/r_0)^{m-2}},
  \end{equation*}
  which readily gives us
  \begin{equation}\label{eq:r-asmyp}
    \log 1/r_0 = \log n - (m-1) \log \log n 
    - \log \frac{f_m(1)}{(m-1)!} + \fOh{\frac{\log \log n}{\log n} }
  \end{equation}
  for $n\to\infty$. Now it follows that
  \begin{equation*}
    nr_0 = \frac{f_m(1)}{(m-1)!} (\log n)^{m-1} 
    \left(1 + \fOh{\frac{\log \log n}{\log n} } \right),
  \end{equation*}
  and Lemma~\ref{lem:asy_central} also yields
  \begin{align*}
    f(r_0,1) &= \frac{f_m(1)}{m!} (\log 1/r_0)^m 
    + \frac{f_{m-1}(1)}{(m-1)!} (\log 1/r_0)^{m-1} + \fOh{(\log n)^{m-2} } \\
    &= \frac{f_m(1)}{m!} (\log n)^m 
    \left( 1 - \frac{m(m-1)}{\log n} \log\log n 
      - \frac{m}{\log n} \log \frac{f_m(1)}{(m-1)!} 
      + \fOh{\frac{\log \log n}{(\log n)^2} } \right) \\
    &\phantom{=}\; + \frac{f_{m-1}(1)}{(m-1)!} (\log n)^{m-1}
    \left( 1 + \fOh{\frac{\log\log n}{\log n}} \right)
    + \fOh{(\log n)^{m-2} \log \log n }.
  \end{align*}
  Since $f_m(1)/m! = \kappa$ and $f_{m-1}(1)/(m-1)! = \kappa m (\sum_{j=1}^m
  \log p_j - \log d)/2$, this readily proves Theorem~\ref{thm:asy-general}
  (note that the factor $f_{tt}(r_0,1)$ only contributes $\fOh{\log n}$ to $\log
  P(n)$).

  To get the more precise formula (Theorem~\ref{thm:asy-2}) in the case $m=2$,
  we only need to expand a little further.
\end{proof}

In principle, it would be possible to obtain similar (as in
Theorem~\ref{thm:asy-2}), more precise asymptotic formulas (in terms of $\log
n$ and $\log \log n$) for all $m \geq 2$, but the expressions become very
lengthy.


\section{Sum of Digits, Hamming Weight, Occurrences of a Digit}
\label{sec:parameters}


This section is devoted to the central limit theorems for the sum
of digits (Theorem~\ref{thm:sum_of_digits}), the Hamming weight
(Theorem~\ref{thm:hamming}) and the occurrence of a fixed digit
(Theorem~\ref{thm:digits}). We will only present the proof for the sum of digits; the other two proofs being analogous. The weak convergence to a Gaussian distribution will follow from the following general theorem (see \cite[Theorem IX.13]{Flajolet-Sedgewick:2009:analy} and the comment thereafter, which states that it is sufficient to consider real values of $u$):

\begin{lem}[cf. {\cite[Theorem IX.13]{Flajolet-Sedgewick:2009:analy}}]\label{thm:quasi-power}
Let $X_1,X_2,\ldots$ be a sequence of discrete random variables that only take on non-negative integer values. Assume that, for $u$ in a fixed interval $\Omega$ around $1$, the probability generating function $P_n(u)$ of $X_n$ satisfies an asymptotic formula of the form
\begin{equation*}
  P_n(u) = \exp(R_n(u)) (1+ o(1))
\end{equation*}
uniformly with respect to $u$, where each $R_n(u)$ is analytic in $\Omega$. Assume also that the conditions
\begin{equation*}
  R_n'(1) + R_n''(1) \to \infty\qquad \text{and} \qquad
  \frac{R'''(u)}{(R_n'(1) + R_n''(1))^{3/2}} \to 0
\end{equation*}
hold uniformly in $u$. Then the normalised random variables
\begin{equation*}
  X_n^* = \frac{X_n - R_n'(1)}{(R_n'(1) + R_n''(1))^{1/2}}
\end{equation*}
converge in distribution to a standard Gaussian distribution.
\end{lem}

\begin{proof}[Proof of Theorem~\ref{thm:sum_of_digits}]
We use Lemma~\ref{lem:saddle}. Let $X_n$ be the sum of digits of a random multi-base representation of $n$, and let
\begin{equation*}
  P_n(u) = \frac{[z^n] F(z,u)}{[z^n] F(z,1)}
\end{equation*}
be the associated probability generating function. In the following, we write
$r(u)$ instead of just $r$ to emphasize the dependence on $u$ (of course, $r$
depends on $n$ as well). Moreover, we set $r_0 = r(1)$ as in the previous
section. In view of Lemma~\ref{lem:saddle}, Lemma~\ref{thm:quasi-power} applies
with
\begin{equation*}
  R_n(u) = n (r(u)-r_0) + f(r(u),u) - f(r_0,1) - \frac12 \log f_{tt}(r(u),u)
  + \frac12 \log f_{tt}(r_0,1).
\end{equation*}
We only have to confirm the conditions on the asymptotic behaviour of the
derivatives. It is easy to extend the argument of Lemma~\ref{lem:asy_central}
to obtain
\begin{equation}\label{eq:partial-derivatives}
\frac{\partial^j}{\partial t^j} \frac{\partial^k}{\partial u^k} f(t,u) =
 \begin{cases}
\frac{\partial^k}{\partial u^k} \frac{f_m(u)}{m!} (\log 1/t)^m
+ \fOh{ (\log 1/t)^{m-1}}, & j = 0, \\
(-1)^j (j-1)! \frac{\partial^k}{\partial u^k} \frac{f_m(u)}{(m-1)!}
\frac{(\log 1/t)^{m-1}}{t^j}
+ \fOh{t^{-j} (\log 1/t)^{m-2}}, & j \neq 0,
\end{cases}
\end{equation}
as $t \to 0^+$, uniformly in $u$.  The definition of $r$ by the implicit
equation $n = - f_t(r(u),u)$ allows us to express $r'(u)$ and all higher
derivatives in terms of derivatives of $f$ by means of implicit
differentiation: we have $r'(u) = -f_{tu}(r(u),u)/f_{tt}(r(u),u)$, and so
forth. Thus it is possible to express the derivatives of $R_n$ only in terms of
$f(r(u),u)$ and its partial derivatives, for which we have the aforementioned
asymptotic formula~\eqref{eq:partial-derivatives}.  Putting everything
together, one obtains
\begin{equation*}
  \frac{\partial^k}{\partial u^k} R_n(u)
  = \frac{1}{m!} \left( \frac{\partial^k}{\partial u^k} f_m(u) \right)
  \Bigl( \log \frac{1}{r(u)} \Bigr)^m
  + \fOh{\Bigl( \log \frac{1}{r(u)} \Bigr)^{m-1}}
\end{equation*}
for $k \in \{1,2,3\}$, so (making use of~\eqref{eq:r-asmyp})
\begin{equation*}
  R_n'(1) \sim \frac{f_m'(1)}{m!} (\log 1/r_0)^m
  \sim  \frac{f_m'(1)}{m!} (\log n)^m
  = \frac{d-1}{2m!} \cdot \prod_{j=1}^m \frac{1}{\log p_j} (\log n)^m
\end{equation*}
and likewise
\begin{equation*}
  R_n''(1) \sim \frac{f_m''(1)}{m!} (\log 1/r_0)^m
  \sim  \frac{f_m''(1)}{m!} (\log n)^m
  = \frac{(d-1)(d-5)}{12m!} \cdot \prod_{j=1}^m \frac{1}{\log p_j} (\log n)^m
\end{equation*}
and $R_n'''(u) = \fOh{(\log n)^m}$ uniformly in $u$. Thus the conditions of Lemma~\ref{thm:quasi-power} are satisfied, which proves asymptotic normality of the distribution. However, we still need to verify the asymptotic behaviour of the moments (which is not implied by weak convergence). To this end, we apply the saddle point method once again.

The generating function of the total sum of digits is $F_u(z,1) = \frac{\partial}{\partial u} F(z,u) \big|_{u=1}$, and the mean is given by
\begin{equation*}
  \mu_n = \frac{[z^n] F_u(z,1)}{[z^n]F(z,1)},
\end{equation*}
so we have to determine an asymptotic formula for the coefficients of
$F_u(z,1)$. Cauchy's integral formula,
  \begin{equation*}
    [z^n]F_u(z,1)=\frac{1}{2\pi i}\oint_{\mathcal{C}}F_u(z,1)\frac{dz}{z^{n+1}},
  \end{equation*}
and the change of variables $z = e^{-t} = e^{-(r_0+i\tau)}$ (where $r_0$ satisfies the saddle point equation as before) yields
\begin{equation*}
  [z^n]F_u(z,1)=\frac{1}{2\pi}\int_{-\pi}^{\pi}
  \fexp{nr_0+f(r_0+i\tau,1)+in\tau}f_u(r_0+i\tau,1)d\tau.
\end{equation*}
Thus,
\begin{multline*}
  [z^n]F_u(z,1) - f_u(r_0,1)[z^n]F(z,1) \\
  = \frac{1}{2\pi}\int_{-\pi}^{\pi}\fexp{nr_0+f(r_0+i\tau,1)+in\tau}
  (f_u(r_0+i\tau,1)-f_u(r_0,1))dt.
\end{multline*}
As we have seen in the proof of Lemma~\ref{lem:saddle}, the tails (the parts of
the integral where $\abs{\tau} \geq r(\log 1/r)^{-c}$) are negligible in that they
only contribute an error term that goes faster to $0$ than any power of $\log
1/r$.
So we may focus on the central part, where we expand into a power series
\begin{multline*}
\fexp{nr_0+f(r_0+i\tau,1)+in\tau}(f_u(r_0+i\tau,1)-f_u(r_0,1)) = e^{nr_0 + f(r_0,1) - f_{tt}(r_0,1) \tau^2/2} \\
\times \Bigl( i f_{tu}(r_0,1)\tau - \frac{f_{ttu}(r_0,1)}{2} \tau^2 - i \frac{f_{tttu}(r_0,1)}{6} \tau^3 + \frac{4f_{ttt}(r_0,1)f_{tu}(r_0,1)+f_{ttttu}(r_0,1)}{24} \tau^4 + \cdots \Bigr).
\end{multline*}
We continue in the same way as in the proof of Lemma~\ref{lem:saddle} to
evaluate the integral over the central region asymptotically by making use of
the asymptotic formula \eqref{eq:partial-derivatives}. This eventually gives us
\begin{equation*}
  \mu_n = \frac{[z^n]F_u(z,1)}{[z^n]F(z,1)}
  = f_u(r_0,1) + \frac{f_{tu}(r_0,1)f_{ttt}(r_0,1) 
    - f_{tt}(r_0,1)f_{ttu}(r_0,1)}{f_{tt}(r_0,1)^2}
  + \fOh{(\log 1/r_0)^{-(m-1)}}.
\end{equation*}
Thus in particular
\begin{align*}
\mu_n &= f_u(r_0,1) + \fOh{1} = \frac{f_m'(1)}{m!} (\log 1/r_0)^m + \fOh{ (\log 1/r_0)^{m-1}} \\
&= \frac{\kappa(d-1)}{2 \log d} (\log n)^{m}
    +\fOh{(\log n)^{m-1}\log\log n}.
\end{align*}
We repeat the process with $F_{uu}(z,u) + F_u(z,u)$ in the place of $F_u(z,u)$
to obtain an asymptotic formula for the second moment, which in turn yields
formula
\begin{align*}
\sigma^2_n &= \frac{[z^n](F_{uu}(z,1)+F_u(z,1))}{[z^n]F(z,1)} - \mu_n^2 = 
f_{uu}(r_0,1) + f_u(r_0,1) + \fOh{(\log 1/r_0)^{m-1}} \\
&=  \frac{f_m'(1)+f_m''(1)}{m!} (\log 1/r_0)^m + \fOh{ (\log 1/r_0)^{m-1}} \\
&= \frac{\kappa(d-1)(d+1)}{12 \log d} (\log n)^{m}
    +\fOh{(\log n)^{m-1}\log\log n}.
\end{align*}
for the variance. This completes our proof.
\end{proof}

\section{Conclusion}
\label{sec:conclusion}

We obtained an asymptotic formula for the number of representations of an
integer $n$ in a multi-base system with given bases $p_1$, $p_2$, \ldots, $p_m$, which are equivalent to partitions into elements of the set
\begin{equation*}
  \Se=\lbrace p_1^{\alpha_1}p_2^{\alpha_2}\dots p_m^{\alpha_m} :
  \alpha_i \in \mathbb{N}\cup \{0\} \rbrace.
\end{equation*}
Moreover, we proved central limit theorems for three very natural parameters: the sum of digits (corresponding to the length of a partition), the Hamming weight (corresponding to the number of distinct parts of a partition), and the number of occurrences of a given digit. There are many more parameters that could be studied; to give one further example, the probablilty that the digit associated with a given element $s \in \Se$ in a random multi-base representation of $n$ is equal to $b$ for some $b \in \{0,1,\ldots,d-1\}$ is $1/d$ in the limit as $n \to \infty$, as one would heuristically expect. It is not difficult to adapt our saddle point approach to this problem, the generating function being
\begin{equation*}
  z^{bs} \prod_{\substack{h \in \Se \\ h \neq s}} \frac{1-z^{hd}}{1-z^h}
\end{equation*}
in this case. As it was already mentioned in Section~\ref{sec:existing}, it would also be possible to extend our results to other digit sets.


\renewcommand{\MR}[1]{}
\bibliographystyle{amsplainurl}
\bibliography{multi-bases-asy-full}


\end{document}